\documentclass[12pt]{article}

\usepackage{amsmath,amsthm}
\usepackage{amssymb,latexsym}

\newtheorem{theorem}{Theorem}[section]

\theoremstyle{definition}
\newtheorem{defin}[theorem]{Definition}
\newtheorem{rem}[theorem]{Remark}
\newtheorem{exa}[theorem]{Example}

\numberwithin{equation}{section}

\newcommand{\PSHG}{{\operatorname{PSHG}}}
\newcommand{\PSH}{{\operatorname{PSH}}}
\newcommand{\PSHS}{{\operatorname{PSHS}}}

\newcommand{\CCC}{{\operatorname{C_+}}}
\newcommand{\cl}{{\operatorname{cl}}}
\newcommand{\D}{{\mathbb D}}
\newcommand{\BB}{{\mathbb B}}

\renewcommand{\Bbb}{\mathbb}

\newcommand{\Rn}{{ \Bbb R}^n}

\newcommand{\Rnp}{{ \Bbb R}_+^n}

\newcommand{\C}{{\Bbb  C}}
\newcommand{\Cn}{{\Bbb  C\sp n}}

\newcommand{\vph}{\varphi}
\newcommand{\fu}{{\mathfrak u}}
\newcommand{\fv}{{\mathfrak v}}

\title{Extreme plurisubharmonic singularities}
\author{Alexander Rashkovskii}
\date{}

\begin{document}

\maketitle

\begin{abstract}
A plurisubharmonic singularity is {\it extreme} if it cannot be represented as the sum of non-homothetic singularities. A complete characterization of such singularities is given for the case of homogeneous singularities (in particular, those determined by generic holomorphic mappings) in terms of decomposability of certain convex sets in $\Rn$. Another class of extreme singularities is presented by means of a notion of relative type.
\end{abstract}

\section{Introduction}
Let $C$ be a convex cone of a vector space $V$. A point $v\in C$ us called {\it extreme} if the relation $v=v_1+v_2$ for $v_i\in C$ implies $v_i=\lambda_iv$ with $\lambda_i\ge0$, $i=1,2$. The set of all extreme points plays an important role due to Choquet's representation theorem. The structure of this set depends on the geometry of the cone, to be investigated in each concrete situation.

In complex analysis this task arises in various contexts. Concerning pluripotential theory, we refer here to papers \cite{Lel}, \cite{De16} on extremal currents, and especially to paper \cite{CeThor} where different types of extremal plurisubharmonic functions were considered; in particular, classical single pole pluricomplex Green functions were shown to be extreme. In \cite{CaWi}, Green functions with several poles were considered.

In this note, we work with the cone of {\it plurisubharmonic singularities} at a fixed point on a complex manifold (basically, $0\in\Cn$), that is, the equivalence classes of asymptotics of plurisubharmonic functions at that point. By using the technique of local indicators from \cite{LeRa}, we obtain a necessary and sufficient condition for a 'homogeneous' singularity to be extreme, in terms of decomposability of certain convex sets in $\Rn$ (Theorem~\ref{theo:main}). Another class of extreme singularities, related to plurisubharmonic valuations \cite{BFaJ}, is described in Theorem~\ref{theo:main2} with the help of a notion of relative type introduced in \cite{R7}. Both classes contain the logarithmic singularity, as well as other 'standard' singularities. In Section~4 we apply this to extreme functions.

\section {Plurisubharmonic singularities}
For basics on plurisubharmonic functions, see, e.g., \cite{Kl}.

Let $\PSHG_p$ be the cone of germs of all plurisubharmonic functions at a point $p$ of a complex manifold. A plurisubharmonic germ is {\it
singular} at $p$ if it is not bounded (from below) in any its neighborhood. The asymptotic behavior of a plurisubharmonic function near its singularity point can be very
complicated.

We say that $u\sim v$ if
$u(z)= v(z)+O(1)$ for $z\to p$. The equivalence
class ${\cl}(u)$ is called the {\it plurisubharmonic
singularity} of $u$ \cite{R9} (in \cite{Zah}, a closely related object was introduced under the name {\it standard singularity}). The collection of all plurisubharmonic singularities
at $p$ is denoted by $\PSHS_p$. Until the last section, we assume $p=0\in\Cn$.

Plurisubharmonic singularities form a convex cone whose extreme rays we will study.

\subsection{Characteristics of singularities}\label{subs:char}

A fundamental characteristic of a singularity $\fu\in
\PSHS_0$ is its {\it Lelong number}
$$ \nu(\fu)= \liminf_{z\to 0}\frac{u(z)}{\log|z|}$$
for any $u\in\fu$ (it is independent of the choice of the representative).  If $f$ is a holomorphic function, then
$\nu(\log|f |)$ equals the multiplicity of $f$ at $0$.

\medskip
A refined version, due to Kiselman \cite{Kis2} (see also \cite{Kis3}), is
the {\it directional Lelong number} in a
direction $a\in\mathbb R_+^n$ (that is, $a_1,\ldots,a_n>0$),
$$\nu(\fu,a)= \liminf_{z\to 0}\frac{u(z)}{\phi_a(z)},\quad u\in\fu,$$
where
\begin{equation}\label{phi}
\phi_a(z)=\max_k \,a_k^{-1}\log|z_k|.\end{equation} In particular, $\nu(\fu)=\nu(\fu,(1,\ldots,1))$.

For polynomials or, more generally, analytic functions $f=\sum c_Jz^J$,
it can be computed as
\begin{equation}\label{index}\nu(\log|f|,a)=\inf\{\langle a,J\rangle:c_J\neq 0\},\end{equation}
the expression in the right-hand side being known in number theory
as the {\it index} of $f$ with respect to the weight $a$, while in commutative algebra it is called  a {\it monomial valuation}.

\medskip
Even more general characteristic was introduced in \cite{R7}.  Recall that an isolated singularity $\vph\in\PSHS_0$ is called {\it maximal} if there exists a representative that is a maximal plurisubharmonic function on a punctured neighborhood of $0$. A {\it relative type} of $\fu$ with respect to a maximal singularity $\vph$ is
\begin{equation}\label{eq:rtype}\sigma(\fu,\varphi)=\liminf_{z\to 0}\frac{u(z)}{\varphi(z)}, \quad u\in\fu.\end{equation}
Its counterpart in algebra (in case of both $\fu$ and $\vph$ with algebraic or analytic singularity) is {\it asymptotic Samuel function}. Note that $\sigma(\fu,\phi_a)=\nu(\fu,a)$.

The relative type gives an upper bound for any $u\in\fu$:
\begin{equation}\label{eq:rbound} u\le\sigma(\fu,\vph)\vph+O(1).\end{equation}

\subsection{Indicators and Newton polyhedra}

The function $t\mapsto\psi_\fu(t)=-\nu(\fu,-t)$, $t\in\mathbb R_-^n=-\mathbb R_+^n$,
is convex and increasing in each $t_k$, so
$\psi_\fu(\log|z_1|,\ldots,|z_n|)$ can be extended (in a unique way)
to a function $\Psi_\fu(z)$ plurisubharmonic in the unit polydisk
$\D^n\subset\Cn$, the {\it (local) indicator} of $\fu$ at $0$, see \cite{LeRa}. Observe
that
\begin{equation}\label{eq:psi+}
\Psi_{\fu+ \fv}=\Psi_\fu + \Psi_\fv.
\end{equation}

The indicators have the log-homogeneity
property
$$\Psi_\fu(z_1,\ldots,z_n)=\Psi_\fu(|z_1|,
\ldots,|z_n|)= c^{-1}\Psi_\fu(|z_1|^c,\ldots,|z_n|^c) \quad \forall
c>0,$$
and any nonpositive plurisubharmonic function $\Phi$ in $\D^n$ with this property is called an {\it indicator},
which is justified by the relation $\Psi_\Phi=\Phi$. The collection of the indicators constitutes a convex cone.

The homogeneity implies $(dd^c\,\Psi_\fu)^n=0$
on $\{\Psi_\fu>-\infty\}$, so if $\Psi_\fu$ is locally bounded outside
$0$, then $(dd^c\,\Psi_\fu)^n=N_\fu\delta_0$ for some $N_\fu\ge 0$ (the {\it Newton number} of $\fu$), and $N_\fu=0$ if
and only if $\Psi_\fu\equiv 0$ ($\delta_0$ being Dirac's
$\delta$-function at $0$).

The indicators are plurisubharmonic characteristics of plurisubharmonic singularities:
\begin{equation}\label{bound} u(z)\le \Psi_u(z)+O(1).\end{equation}
When $u$ has isolated singularity at $0$, this implies (by
Demailly's comparison theorem \cite{D}) a relation between the
Monge-Amp\`ere measures:
$$(dd^cu)^n\ge (dd^c\,\Psi_u)^n=N_u\delta_0.$$

\medskip

Due to the homogeneity, the convex image
$\psi_\fu(t)=\Psi_\fu(e^{t_1},\ldots,e^{t_n})$ of
the indicator $\Psi_\fu$ coincides with the support function to the
convex set
$$\Gamma_\fu=\{b\in\overline{\mathbb R_+^n}:\: \psi_\fu(t)\ge\langle b,t \rangle\
\forall t\in\mathbb R_-^n\},$$  that
is, $$\psi_\fu(t)=\sup\,\{\langle t, a\rangle: \: a\in \Gamma_\fu\}.$$
We will call the set $\Gamma_\fu$ by {\it indicator diagram} of $\fu$. For $\fu=\cl(\log|f|)$ this is precisely the Newton polyhedron of the function $f=\sum c_Jz^J$ at $0$,  i.e., the convex hull of the set $\{J+\Rnp:\: c_J\neq 0\}$, see (\ref{index}).

\medskip

Let $\CCC$ be the collection of all closed convex subsets $\Gamma$ of $\mathbb R_+^n$ that are {\it complete} in the following sense: $a\in\Gamma\Rightarrow a+\mathbb R_+^n\subset\Gamma$.
We have just established an isomorphism between the cone of
the indicators and the cone $\CCC$ endowed with Minkowski's addition
$$\Gamma_1+\Gamma_2=\{a+b:\: a\in\Gamma_1,\ b\in\Gamma_2\}.$$
By (\ref{eq:psi+}) and the corresponding property of the support function,
\begin{equation}\label{eq:sumgamma}\Gamma_{\fu+ \fv}=\Gamma_\fu + \Gamma_\fv.\end{equation}
Note also that the Newton number $N_\fu$ of an isolated singularity $\fu$ can be computed as  \begin{equation}\label{eq:volume} N_\fu=n!\,{\rm Vol}(\Rnp\setminus\Gamma_\fu),\end{equation}
see \cite{R}.

\section{Extreme singularities}

\begin{defin} We say that a singularity $\fu\in\PSHS_0$ is {\it extreme} if the relation $\fu=\fu_1+\fu_2$ for $\fu_i\in\PSHS_0$ implies $\fu_i=\lambda_i\fu$ with $\lambda_i\ge0$.
\end{defin}

In terms of germs, this means that the relation $u=u_1+u_2+ O(1)$ for $u_i\in\PSHG_0$ implies $u=\lambda_iu_i+O(1)$ with $\lambda_i\ge0$.

\medskip

\subsection{Indicator diagram test}

We are going to check the singularities by means of their indicator diagrams.

\medskip

\begin{defin} A set $K\subset\CCC$ is called {\it decomposable} if there exist sets $K_1, K_2\in\CCC$, non-homothetic to $K$, such that $K=K_1+K_2$. (A set $A\in\CCC$ is {\it homothetic} to $B$ if $A=cB+x$ for $c\ge 0$ and $x\in\Rnp$.)
\end{defin}

For the case of arbitrary convex polyhedra in $\Rn$, this notion has been extensively studied, see e.g. \cite{Sha}, \cite{She}, \cite{Mey}, \cite{Kal}, \cite{Mcm}, \cite{Smi} where a number of results on (in)decomposability of polyhedra are obtained. Decomposability of Newton polyhedra (with application to reducibility of polynomials and analytic functions) was considered, for example, in \cite{Sha}, \cite{Lip}, \cite{GS}. Observe that for such an application one does not exclude homothetic polyhedra, while we have to do that in order to treat extreme singularities. Our definition is thus closer to that from \cite{Smi}.

Note that a polyhedron is decomposable in the class of all convex sets in $\CCC$ if and only if it is decomposable in the class of convex polyhedra in $\CCC$. A straightforward example of an indecomposable set in $\CCC$ is
$$ \Gamma_a=\{x\in\overline{\Rnp}:\: \langle x,a\rangle\ge1\},\quad a\in\Rnp.$$

Relation (\ref{eq:sumgamma}) makes one hope that there should be a strong connection between the extremity and indecomposability. However the things are not that simple.

\begin{exa}\label{ex:Nd}
The Newton diagram of the function
\begin{equation}\label{eq:decomp1} u=\log(|z_1^3| + |z_1^3+z_1^2z_2|+|z_1^2+z_1z_2|+|z_1^2+2z_1z_2+z_2^2|)\end{equation}
is $\Gamma_{(2,2)}$ and therefore is indecomposable. At the same time, $u$ is not extreme, just because $u=\log(|z_1|+|z_2|) + \log(|z_1^2|+|z_2|)$.
\end{exa}

The property of being extreme is obviously coordinate independent, while Newton polyhedra are very sensitive to the choice of coordinates, For instance, under the linear transform $\zeta_1=z_1$, $\zeta_2=z_1+z_2$, the function $u$ for Example~\ref{ex:Nd} turns to
\begin{equation}\label{eq:decomp2} v(\zeta)= \log(|\zeta_1^3| + |\zeta_1^2\zeta_2|+|\zeta_1\zeta_2|+|\zeta_2^2|)\end{equation}
whose indicator diagram $\Gamma_v$ is
generated by the points $(3,0)$, $(1,1)$, and $(0,2)$, so it equals the sum $\Gamma_{(1,1)}+\Gamma_{(2,1)}$, none of which being homothetic to $\Gamma_v$.

In addition, we refer to the well-known problem of existence of isolated singularities $\varphi$ that have zero Lelong number but nonzero residual Monge-Amp\`ere mass at $0$. The indicator of such a singularity is identical zero, so the indicator diagram of any $\fu\in\PSHS_0$ coincides with that of the non-extreme function $u+\varphi$.

\medskip

To avoid these problems, we restrict ourselves to a subclass of the singularities.
According to \cite{R4}, a function $u\in\PSHG_0$ is {\it almost homogeneous} if $\Psi_u\in \cl(u)$, that is, when the inequality in (\ref{bound}) becomes an equality. This means that one can always find a homogeneous representative $\Psi_u$ of $\cl(u)$, so we call it a {\it homogeneous singularity}.

As was proved in \cite{R7}, a function $u\in\PSHG_0$ with isolated singularity is almost homogeneous if and only if its residual mass $(dd^cu)^n(0)$ coincides with that of its indictor. By Kouchnirenko's theorem (\cite{Ku1}, \cite{AYu}), the latter is true for $u=\log|F|$, where $F$ is a generic holomorphic mapping with a given Newton polyhedron. Other examples of almost homogeneous functions can be found in \cite{R}, \cite{R4}.

\begin{theorem}\label{theo:main}
A homogeneous singularity is extreme if and only if its indicator diagram is indecomposable.
\end{theorem}

\begin{proof} Let ${\mathfrak u}$ be a homogeneous singularity and $\Psi$ be its indicator representative. Then $\Psi=u_1+u_2+O(1)$ with $u_i\in\PSHG_0$ if and only if
\begin{equation}\label{eq:repr} \Psi=\Psi_{u_1}+ \Psi_{u_2}.\end{equation}
Since $u_i\le\Psi_{u_i}+O(1)$, one has then $u_i=\Psi_{u_i}+O(1)$, so the singularities $\cl(u_i)$ are homogeneous as well. Therefore, $\fu$ is extreme if and only if the representation (\ref{eq:repr}) is possible with $\Psi_{u_i}=\lambda_i\Psi$ only, which exactly means that $\Gamma_\fu$ is indecomposable.
\end{proof}

\medskip

\begin{rem} The function $v$ defined by (\ref{eq:decomp2}) is almost homogeneous because $v\ge\Psi_v$, so the decomposability of its indicator diagram reflects perfectly the non-extremity of $v$. In contrast to that, the function $u$ from (\ref{eq:decomp1}) has indecomposable indicator diagram, however it is not extreme. This is caused by the fact that $u$ is not almost homogeneous, which can be checked by a direct computation of the residual Monge-Amp\`ere masses by means of (\ref{eq:volume}): $(dd^cu)^n(0)=(dd^cv)^n(0)=(dd^c\Psi_v)^n(0)=6$, while $(dd^c\Psi_u)^n(0)=5$.
\end{rem}

\subsection{Additive types}

Another class of extreme singularities comes from the notion of relative type (\ref{eq:rtype}). As follows from the definition, type with respect to any maximal singularity $\vph$ satisfies $\sigma(\sum\fu_i,\vph)\ge \sum\sigma(\fu_i,\vph)$.

\begin{defin} We will say that a maximal singularity $\vph$ is {\it additive} if
$$ \sigma\left(\sum\fu_i,\vph\right)= \sum\sigma(\fu_i,\vph)\quad {\forall } \fu_i\in\PSHS_0.$$
\end{defin}

For example, {\it flat weights} considered in \cite{R8} possess this property; in particular, such are the simplicial singularities $\phi_a$ and $\vph=\log(|z_1|^s+|f|)$ in $\C^2$, $s>0$, where $f$ is any irreducible holomorphic function whose zero set is transverse to $\{z_1=0\}$ and the multiplicity at $0$ is at most $s$, see \cite{FaJ}. More generally, all plurisubharmonic weights generating {\it quasimonomial valuations} on $\Cn$, see \cite{BFaJ} and \cite{R9}, are additive.

\begin{theorem}\label{theo:main2} Any additive maximal singularity is extreme.
\end{theorem}

\begin{proof} Let $\vph$ be a fixed representative of the given additive maximal singularity, and assume
\begin{equation}\label{eq:summa}
\vph=u_1+u_2+O(1)
\end{equation}
Denote $\sigma_i=\sigma(u_i,\vph)$, then the additivity gives us
$\sigma_1+\sigma_2=1$. If $\sigma_1=0$, then $\sigma_2=1$ and the bound (\ref{eq:rbound}) implies $u_2\le\vph+O(1)$. In view of (\ref{eq:summa}) we have, in addition, $\vph\le u_2+O(1)$, so $u_2\in\cl(\vph)$, which proves the assertion for this case.

Now we can assume $\sigma_1>0$ and $\sigma_2=1-\sigma_1>0$. Denote $$v=\max_i\frac{u_i}{\sigma_i},$$
then $\sigma(v,\vph)=1$, so $v\le\vph+O(1)$. On the other hand,
$$v=\max_i\frac{u_i}{\sigma_i}\ge \sigma_1\frac{u_1}{\sigma_1}+(1-\sigma_1)\frac{u_2}{\sigma_2}=u_1+u_2=\vph+O(1),$$
so $v\in\cl(\vph)$. We claim that this implies $u_1/\sigma_1=u_2/\sigma_2+O(1)$. Assuming the contrary, there exists a sequence of points $z_k\to 0$ such that, for example,  $u_1(z_k)/\sigma_1-u_2(z_k)/\sigma_2=A_k\to\infty$. Therefore,
$$v(z_k)=\frac{u_1(z_k)}{\sigma_1}=\frac{\sigma_1u_1(z_k)+\sigma_2u_1(z_k)}{\sigma_1}=u_1(z_k)+u_2(z_k)+\sigma_2A_k, $$
which contradicts $v\in\cl(\vph)$.

Therefore, $v=u_i/\sigma_i +O(1)$, so $u_i\in\sigma_i\,\cl(\vph)$.
\end{proof}

\section{Extreme plurisubharmonic functions}

The pluricomplex Green function of a bounded hyperconvex domain $\Omega$ for a maximal singularity $\fu\in\PSHS_p$ was introduced in \cite{Zah} (in the case of continuous singularity) and in \cite{R7} (in the general case) as
$$G_{\fu}=\sup\{v\in\PSH^-(\Omega):\: v\in\fu\}.$$
It is maximal on $\Omega\setminus\{p\}$ and $G_{\fu}\in\fu$, and this is the unique plurisubharmonic function with these properties.

When the singularity is homogeneous with a representative $u=\Psi(\cdot+p)$ for a given indicator $\Psi$, this coincides with the function introduced in \cite{LeRa} as the upper envelope of negative plurisubharmonic functions $v$ in $\Omega$ such that $\Psi_{v(\cdot-p)}\le \Psi$.
 When $\Psi=\log|z|$, this produces the standard pluricomplex Green functions $G_p$ with pole at $p$.

In \cite{CeThor}, the classical pluricomplex Green functions were shown to be extreme: $G_p=u_1+u_2$ for $u_1,u_2\in\PSH^-(\Omega)$ implies $u_i=\lambda_iG_p$, $\lambda_i\ge 0$.

\begin{theorem} The pluricomplex Green function for an extreme maximal singularity $\fu$ at $p\in\Omega$ is extreme. When the singularity is homogeneous, $\Omega=\D^n$ and $p=0$, the converse is true as well.
\end{theorem}

\begin{proof} Let $G_{\fu}=u_1+u_2$, then each $u_j$ is maximal on $\Omega\setminus\{p\}$ and equals $0$ on $\partial\Omega$, which implies $u_i=G_{\cl(u_i)}$. In addition, $\cl(u_1)+\cl(u_2)=\fu$, so $\cl(u_i)=\lambda_i\fu$, and again the uniqueness theorem yields $G_{\cl(u_i)}=\lambda_iG_\fu$.

The second assertion follows immediately from Theorem~\ref{theo:main} and the observation that, in the case of $p=0\in\D^n=\Omega$, one has $G_\Psi=\Psi$.
\end{proof}

\medskip

For an arbitrary domain $\Omega$, non-extremity of a singularity does not imply non-extremity of the Green function. For example, the indicator of the function $v$ defined by (\ref{eq:decomp2}) (i.e., its Green function for $\D^2$) is $$\Psi_v(z)=\max\{3\log|z_1|, 2\log|z_1|+\log|z_2|,\log|z_1|+\log|z_2|,2\log|z_2|\}$$ and it represents as the sum of $\Psi_1$ and $\Psi_2$ with
$ \Psi_1(z)=\max\{\log|z_1|,\log|z_2|\}$ and $ \Psi_2(z)=\max\{2\log|z_1|,\log|z_2|\}$.

On the other hand, let $G_v$ be its Green function in the unit ball $\BB_2$. Assume that $G_v=v_1+v_2$ for $v_i\in\PSH^-(\BB_2)$. Then $\Psi_v=\Psi_{v_1}+\Psi_{v_2}$ and we again, as in the proof of Theorem~\ref{theo:main}, conclude that both $v_1$ and $v_2$ are almost homogeneous (since $v$ is so) and moreover, since $\Gamma_v$ has a unique decomposition into the sum of two diagrams (which checks directly), $\Psi_{v_1}$ equals either $\Psi_1$ or $\Psi_2$. Therefore, $$G_{\Psi_{v_i}}= G_{v_i}\ge v_i,$$
so $G_v=G_{\Psi_1}+G_{\Psi_2}$. Substituting here $G_{\Psi_1}=\log|z|$ and a known formula for $G_{v_2}$ from \cite{RSig2}, we see that the sum does not satisfy the homogeneous Monge-Amp\`ere equation outside $0$.

\medskip

This observation leads us to the following

\medskip

{\rm\bf Conjecture.} {\sl Any plurisubharmonic solution to the Dirichlet problem in the unit ball $\BB_n$ $$(dd^c u)^n=\delta_p, \quad u|_{\partial\BB_n}=0,$$ is an extreme plurisubharmonic function.}

\vskip1cm

Tek/Nat, University of Stavanger, 4036 Stavanger, Norway

\vskip0.1cm

{\sc E-mail}: alexander.rashkovskii@uis.no

\end{document}